\documentclass[10pt, a4paper]{amsart}
\usepackage{amssymb,amsmath,euscript,mathrsfs,bbm,verbatim,txfonts}
\newtheorem*{theorem}{Theorem}
\newtheorem*{lemma}{Lemma}
\newtheorem*{proposition}{Proposition}

\newtheorem*{claim}{Claim}
\title{Menshov's "adjustment theorem" with respect to general measures}
\author{Themis Mitsis}
\address{Department of Mathematics and Applied Mathematics, University of Crete, 70013 Heraklion, Greece}
\email{themis.mitsis@gmail.com}
\begin{document}
\begin{abstract}
We prove Menshov's theorem in the setting of  arbitrary Borel measures.
 \end{abstract}
\maketitle
\pagestyle{plain}
A classical theorem of Menshov asserts that a Lebesgue measurable real function on $[0,2\pi]$  can be modified on a set of arbitrarily small Lebesgue measure so that the resulting function  is continuous and has uniformly convergent Fourier series. Menshov's intricate proof appeared in \cite{menshov}. K\"orner gave a different proof in \cite{kor1}. A highly readable account of Menshov's original proof may be found in Bary \cite{bary}.

Note that if we replace the requirement that the modified function should have uniformly convergent Fourier series with the requirement that it should be merely continuous, then we get Lusin's theorem which easily  holds with any positive, finite Borel measure in place of Lebesgue measure. So it is natural to ask whether Menshov's theorem holds for arbitrary measures as well. The purpose of this paper is to answer this question in the affirmative. Namely we prove the following.

\begin{theorem}
 If $f:[0,2\pi]\to\mathbbm R$ is Borel measurable and $\mu$ is a positive, finite Borel measure on $[0,2\pi]$, then for every $\varepsilon>0$ there exists a continuous $g:[0,2\pi]\to\mathbbm R$ with uniformly convergent Fourier series so that $\mu(\{f\ne g\})<\varepsilon$.
\end{theorem}

We first note that this is plainly true if $\mu$ is purely atomic, that is if $\mu=\sum_{n=1}^{+\infty}a_n\delta_{x_n}$, where $\delta_{x_n}$ are Dirac deltas and $a_n$ are positive numbers with $\sum_na_n<+\infty$. So it suffices to consider the case where $\mu$ is non-atomic, that is $\mu(\{x\})=0$ for all $x$.

Let us very briefly go through Menshov's proof to see what the difficulty is if we try to replace Lebesgue measure with an arbitrary measure. By Lusin's theorem we can assume that $f$ is continuous. Then we approximate $f$ with step functions. The heart of the proof is the following lemma on page 500 in \cite{bary}.
\begin{lemma}[Menshov's Lemma]
 Let $[c,d]\subset[0,2\pi]$, $\gamma\in\mathbbm R$, $\varepsilon>0$, and $\nu\in\mathbbm N$ with $\nu>8$. Then there exists a continuous piecewise linear function $\psi$ supported in $[c,d]$, a Borel set $E\subset[c,d]$, and an absolute constant $B>0$ so that
\begin{enumerate}
  \item $|\psi|\leq2\nu|\gamma|$.
  \item $\psi=\gamma$ in $E$.
  \item $\displaystyle\left|\int_0^\xi\psi\right|<\varepsilon$, for all $\xi\in[0,2\pi]$.
  \item $\displaystyle\left|\int_0^{2\pi}\psi(t)\frac{\sin j(t-x)}{t-x}\, dt\right|<B\nu|\gamma|$, for all $j\in\mathbbm N$, $x\in[0,2\pi]$.
\end{enumerate}
The set $E$ is constructed as follows. We take $r$ to be an integer so that $q=r\nu$ satisfies \[4|\gamma|(d-c)/q<\varepsilon\] and we let
\begin{itemize}
 \item $\delta=\frac{d-c}{qn}$.
 \item $c_s=c+\delta\nu s$, $s=0,1,\dots,q$, $c_0=c$, $c_q=d$.
 \item $a_s=c_s-\delta$.
 \item $a'=c+2\frac{d-c}{\nu}$, $b'=d-2\frac{d-c}{\nu}$.
\end{itemize}
Then
\[E=[a',b']\smallsetminus\bigcup_{s=2r+1}^{q-2r}[a_s,c_s].\]
The Lebesgue measure of $E$ is at least $(d-c)(1-\frac5\nu)$, and so for large $\nu$ it can be made as close to the measure of $[c,d]$ as we please.
 \end{lemma}

By means of Menshov's lemma, a step function is modified on a set of small measure to get a suitable continuous and piecewise linear function. Combining such functions we can complete the proof. The "small" set is a finite union of sets like the complement of the set $E$ in Menshov's lemma. This set is a finite union of intervals whose endpoints are arithmetic progressions and, of course, its Lebesgue measure can be trivially calculated by just adding up lengths. However, it is not obvious how to estimate its $\mu$-measure. Surprisingly, the following result shows that such a set inside an interval, asymptotically occupies the same proportion of the interval's $\mu$-measure as its Lebesgue measure.

\begin{proposition}
Let $\mu$ be a positive, finite, non-atomic Borel measure on $[0,2\pi]$, $I=[a,b]\subset[0,2\pi]$, $n\in\mathbbm N$ and $\sigma,\tau$ positive numbers so that $\sigma+\tau<1$.
Define
\[A_n=\bigcup_{k=0}^{n-1}\left[a+(k+\sigma)\frac{b-a}n,a+(k+\sigma+\tau)\frac{b-a}n\right].\]
Such a set is called an M-set. Then for every $m\in\mathbbm N$ there exists a subset $\widetilde\Lambda\subset m\mathbbm N$ such that
\[\lim_{n\in\widetilde\Lambda}\mu(A_n)=\tau\mu(I).\]
\end{proposition}

\begin{proof}
For any real number $x$ let $(x)=x-[x]$, where $[\cdot]$ is the integer part. Note that if $\ell$ is the affine map which takes $[0,1]$ onto $I$, then 
\[A_n=\ell(B_n)\cap I,\]
where \[B_n=\left\{x:(nx)\in[\sigma,\sigma+\tau]\right\}.\]
We normalize $\mu$ by defining the measure
\[\nu(E)=\frac{\mu(\ell(E)\cap I)}{\mu(I)}.\]
If 
\[\widehat{\nu}(j)=\int e^{-2\pi ijt}\, d\mu(t),\ j\in\mathbbm Z,\]
then by Wiener's theorem, since $\nu$ is non-atomic, we get that
\[\lim_{N\to+\infty}\frac1{N+1}\sum_{n=0}^{N}\left|\, \widehat{\nu}(nk)\right|^2=0,\]
for every $k$. Now let
\[\Lambda_{j,k}=\left\{n:\left|\,\widehat{\nu}(nk)\right|\leq\frac1j\right\}.\]
Then 
\[\frac{\left|\Lambda_{j,k}^\complement\cap[0,N]\right|}{N+1}\leq\frac{j^2}{N+1}\sum_{n=0}^N\left|\,\widehat{\nu}(nk)\right|^2.\]
Cosequently, $\Lambda_{j,k}$ has density $1$, therefore there exists a set $\Lambda_k$ with density $1$ so that $\Lambda_k\smallsetminus\Lambda_{j,k}$ is finite. Hence
\[\lim_{n\in\Lambda_k}\widehat\nu(nk)=0,\]
for all $k$. We conclude that there is a set $\Lambda$ with density $1$ so that
\[\lim_{n\in\Lambda}\widehat\nu(nk)=0,\]
for all $k$. Now if $\mathbbm T=\{z\in\mathbbm C:|z|=1\}$ is the unit circle, we define $F_n:\mathbbm R\to\mathbbm T$ by
\[F_n(x)=e^{2\pi inx}.\]
Then 
\[\lim_{n\in\Lambda}\int F_n^k(x)\, d\nu(x)=\lim_{n\in\Lambda}\widehat\nu(nk)=0.\]
for all $k$. Next, let $\lambda$ be normalized Lebesgue measure on $\mathbbm T$, and for every Borel set $E\subset\mathbbm T$, let 
\[P_n(E)=\nu(F_n^{-1}(E))\]
be the distribution of $F_n$. But then
\[\int F_n^k\, d\nu(x)=\int_{\mathbbm T}z^k\, dP_n(z),\]
so 
\[\lim_{n\in\Lambda}\int_{\mathbbm T}z^k\, dP_n(z)=0,\]
for all $k$, and therefore
\[\lim_{n\in\Lambda}\int_{\mathbbm T}\varphi(z)\, dP_n(z)=0,\]
for every trigonometric polynomial $\varphi$. We conclude that
\[\lim_{n\in\Lambda}\int_{\mathbbm T}f(z)\, dP_n(z)=\int_{\mathbbm T}f(z)\, d\lambda(z),\]
for every continuous function $f$ on $\mathbbm T$. Consequently
\[\lim_{n\in\Lambda}P_n(E)=\lambda(E)\]
for every $E$ whose boundary has Lebesgue measure zero. In particular, if
\[E=\left\{e^{2\pi it}:t\in[\sigma,\sigma+\tau]\right\},\]
we obtain
\[\nu(B_n)=P_n(E)\underset{n\in\Lambda}{\longrightarrow}\lambda(E)=\tau.\]
To finish the proof, we let $\widetilde\Lambda=\Lambda\cap m\mathbbm N$.
\end{proof}
We are now in a position to prove Menshov's theorem for the measure $\mu$. A careful reading of the argument on pages 506-507 in \cite{bary} shows that it is enough to prove the following.

\begin{claim}
Let $\varphi:[0,2\pi]\to\mathbbm R$ be a step function, $\nu\in\mathbbm N$ with $\nu>8$, and $\varepsilon_k$ a positive sequence. Then there exist 
\begin{itemize}
\item a finite partition of $[0,2\pi]$ into disjoint intervals $J_k=[\hat c_k,\hat d_k]$,  so that $\varphi$ is constant on $J_k$ and $\varphi|_{J_k}=\gamma_k$, 
\item continuous piecwise linear functions $\psi_k$ supported in $J_k$,
\item a Borel set $E\subset[0,2\pi]$ with $\mu(E)\geq(1-7/\nu)\mu([0,2\pi])$ 
\end{itemize}
such that
\begin{enumerate}
  \item $|\psi_k|\leq2\nu|\gamma_k|$, for all $k$.
  \item $\psi_k=\gamma_k$ in $E$.
  \item $\displaystyle\left|\int_0^\xi\psi_k\right|<\varepsilon_k$, for all $\xi\in[0,2\pi]$ and all $k$.
  \item $\displaystyle\left|\int_0^{2\pi}\psi_k(t)\frac{\sin j(t-x)}{t-x}\, dt\right|<B\nu|\gamma_k|$, for all $j\in\mathbbm N$, $x\in[0,2\pi]$ and all $k$.
\end{enumerate}
\end{claim}

To prove the claim we can assume that the step function $\varphi$ is constant on $\rho$ intervals of equal length. We further subdivide these intervals to take $\rho\kappa$ intervals of equal length. These are the $J_k$'s. $\kappa$ will be determined later. We apply Menshov's lemma to the triple  $J_k$, $\gamma_k$, $\varepsilon_k$ to get the function $\psi_k$ and a set $E_k\subset[a_k',b_k']\subset I_k$, where
\[a_k'=\hat c_k+2\frac{\hat d_k-\hat c_k}{\nu},\ b'=\hat d_k-2\frac{\hat d_k-\hat c_k}{\nu}.\]
Note that $[a_k',b_k']\smallsetminus E_k$ is an M-set in  $[a_k',b_k']$ with $\tau=\frac1\nu$, consisting of $(\nu-4)r$ intervals ($r$ is as in the statement of Menshov's lemma). Therefore by our Proposition, if   $r$ is large enough then 
\[\mu(E_k)\geq\left(1-\frac2\nu\right)\mu([a_k',b_k']).\]
Now we let 
\[E=\bigcup_{k}E_k.\]
Then 
\[\mu(E)\geq\left(1-\frac2\nu\right)\mu\left(\bigcup_k[a_k',b_k']\right)\]
$\bigcup_k[a'_k,b_k']$ is an M-set in $[0,2\pi]$ with $\tau=1-\frac4\nu$, consisting of $\rho\kappa$ intervals. So again by our Proposition, for $\kappa$ large enough we have
\[\mu\left(\bigcup_k[a_k',b_k']\right)\geq\left(1-\frac5\nu\right)\mu([0,2\pi]).\]
We conclude that
\[\mu(E)\geq\left(1-\frac7\nu\right)\mu\left([0,2\pi]\right).\]


\begin{thebibliography}{99}
\bibitem{bary} N. K. Bary. A treatise on trigonometric series. Vol. 1. The Macmillan Company, New York (1964).
\bibitem{kor1} T. W. K\"orner. A theorem of Menšov on the adjustment of functions. {\it J. Lond. Math. Soc.}, II. Ser. 60, No.2, 548-560 (1999).
\bibitem{menshov} D. E. Menshov. Sur la convergence uniforme des series de Fourier. {\it Matematicheskii Sbornik}, 11(53), 69-76 (1942).
\end{thebibliography}
\end{document}